\documentclass[12pt]{amsart}

\usepackage{pinlabel}
\usepackage{amsmath,amsfonts,amssymb,latexsym,mathrsfs,stmaryrd}
\usepackage{color}
\usepackage{graphicx}
\usepackage{graphics}
\usepackage{enumerate}
\usepackage{amscd} 
\usepackage{amsxtra}
\usepackage{epic,eepic}
\usepackage{hyperref} 
\usepackage{enumitem}
\usepackage{MnSymbol}
\usepackage{dsfont}
\usepackage[all]{xy}

\newtheorem{theorem}{Theorem}[section]

\newtheorem{proposition}[theorem]{Proposition}
\newtheorem{corollary}[theorem]{Corollary}

\newtheorem{lemma}[theorem]{Lemma}

\theoremstyle{definition}

\newtheorem{example}[theorem]{Example}

\newtheorem{remark}[theorem]{Remark}

\newtheorem{question}[theorem]{Question}
 
\newtheorem{conv}[theorem]{Convention}
\newcommand{\qu}{/\kern-.7ex/}
\newcommand{\lqu}{\backslash \kern-.7ex \backslash}
\newcommand{\on}{\operatorname} 
\newcommand{\Aut}{\on{Aut}}

\title[Relative and Orbifold GW]{Higher genus relative and orbifold Gromov-Witten invariants}

\author{Hsian-Hua Tseng}
\address{Department of Mathematics\\ Ohio State University\\ 100 Math Tower, 231 West 18th Ave.\\Columbus\\ OH 43210\\ USA}
\email{hhtseng@math.ohio-state.edu}

\author{Fenglong You}
\address{Department of Mathematical and Statistical Sciences\\ 632 CAB \\ University of Alberta\\Edmonton\\ AB\\ T6G 2G1\\ Canada}
\email{fenglong@ualberta.ca}

\thanks{}


\begin{document}
\date{\today}

\begin{abstract} 
Given a smooth projective variety $X$ and a smooth divisor $D\subset X$. We study relative Gromov-Witten invariants of $(X,D)$ and the corresponding orbifold Gromov-Witten invariants of the $r$-th root stack $X_{D,r}$. For sufficiently large $r$, we prove that orbifold Gromov-Witten invariants of $X_{D,r}$ are polynomials in $r$. Moreover, higher genus relative Gromov-Witten invariants of $(X,D)$ are exactly the constant terms of the corresponding higher genus orbifold Gromov-Witten invariants of $X_{D,r}$. We also provide a new proof for the equality between genus zero relative and orbifold Gromov-Witten invariants, originally proved by Abramovich-Cadman-Wise \cite{ACW}. When $r$ is sufficiently large and $X=C$ is a curve, we prove that stationary relative invariants of $C$ are equal to the stationary orbifold invariants in all genera.
\end{abstract}

\maketitle 

\tableofcontents
\section{Introduction}

Gromov-Witten theory associated to a smooth projective variety $X$ is an enumerative theory about counting curves in $X$ with prescribed conditions. Gromov-Witten invariants are defined as intersection numbers on the moduli space $\overline{M}_{g,n,d}(X)$ of n-pointed, genus $g$, degree $d\in H_2(X,\mathbb Z)$, stable maps to $X$. 

Given a smooth divisor $D$ in $X$, one can study the enumerative geometry of counting curves with prescribed tangency conditions along the divisor $D$. There are at least two ways to impose tangency conditions.

\subsection{Relative Gromov-Witten Invariants}
The first way to impose tangency conditions is to consider relative stable maps to $(X,D)$ developed in \cite{IP}, \cite{LR}, \cite{Li}. 

For a degree $d\in H_2(X,\mathbb Z)$, we consider a partition $\vec k=(k_1,\ldots,k_m)\in (\mathbb Z_{>0})^m$ of $\int_d[D]$. That is, 
\[
\sum_{i=1}^m k_i=\int_d[D].
\]
A {\em cohomology weighted partition} $\mathbf k$ of $\int_d[D]$ is a partition $\vec k$ whose parts are weighted by cohomology classes of $H^*(D,\mathbb Q)$. More precisely,
\[
\mathbf k=\{(k_1,\delta_{1}),\ldots, (k_m,\delta_{m})\},
\]
such that
\begin{itemize}
\item $\sum_{i=1}^m k_i=\int_d[D]$;
\item $\delta_{i}\in H^*(D,\mathbb Q), \quad 1\leq i \leq m$.
\end{itemize}
Cohomology weighted partitions will appear in the degeneration formula for Gromov-Witten invariants.

\begin{conv}
When $X$ is a curve and $D$ is a point, the cohomology weights are just the identity class of $H^*(pt,\mathbb Q)$. In this case, we will not distinguish $\mathbf k$ and $\vec k$. 
\end{conv}

We consider the moduli space $\overline{M}_{g,\vec k,n,d}(X,D)$ of $(m+n)$-pointed, genus $g$, degree $d\in H_2(X,\mathbb Z)$, relative stable maps to $(X,D)$ such that the relative conditions are given by the partition $\vec k$. We assume the first $m$ marked points are relative marked points and the last $n$ marked points are non-relative marked points. Let $\on{ev}_i$ be the $i$-th evaluation map, where
\begin{align*}
\on{ev}_i: \overline{M}_{g,\vec k,n,d}(X,D) \rightarrow D, & \quad\text{for } 1\leq i\leq m;\\
\on{ev}_i: \overline{M}_{g,\vec k,n,d}(X,D) \rightarrow X, & \quad \text{for } m+1\leq i\leq m+n.
\end{align*}

There is a stabilization map 
\[
s:\overline{M}_{g,\vec k, n,d}(X,D)\rightarrow \overline{M}_{g,m+n,d}(X).
\]
Write $\bar\psi_i= s^*\psi_i$ which is the class pullback from the corresponding descendant class on the moduli space $\overline{M}_{g,m+n,d}(X)$ of stable maps to $X$.
Consider
\begin{itemize}
\item $\delta_{i}\in H^*(D,\mathbb Q)$, for $1\leq i\leq m$.
\item $\gamma_{m+i}\in H^*(X,\mathbb Q)$, for $1\leq i\leq n$.
\item  $a_{i}\in \mathbb Z_{\geq 0}$, for $1\leq i\leq m+n$.
\end{itemize}
Relative Gromov-Witten invariants of $(X,D)$ are defined as
\begin{align}\label{relative-invariant-higher-dimension}
&\left\langle \prod_{i=1}^m\tau_{a_{i}}(\delta_{i})\left|\prod_{i=1}^n \tau_{a_{m+i}}(\gamma_{m+i})\right.\right\rangle^{(X,D)}_{g,\vec k,n,d}:=\\
\notag &\int_{[\overline{M}_{g,\vec k,n,d}(X,D)]^{vir}}\psi_1^{a_1}\on{ev}^*_{1}(\delta_{1})\cdots \psi_m^{a_m}\on{ev}^*_{m}(\delta_{m})\psi_{m+1}^{a_{m+1}}\on{ev}^{*}_{m+1}(\gamma_{m+1})\cdots\psi_{m+n}^{a_{m+n}}\on{ev}^{*}_{m+n}(\gamma_{m+n}).
\end{align}
We refer to \cite{IP}, \cite{LR}, \cite{Li} for more details about the construction of relative Gromov-Witten theory.

\subsection{Orbifold Gromov-Witten Invariants}

Another way to impose tangency conditions is to consider orbifold Gromov-Witten invariants of the $r$-th root stack $X_{D,r}$ of $X$ for a positive integer $r$ \cite{Cadman}. By \cite{GS}, root construction is essentially the only way to construct stack structures in codimension one. The construction of root stacks can be found in \cite[Appendix B]{AGV} and \cite{Cadman}. 

\begin{example}
For a positive integer $r$, the $r$-th root stack of $\mathbb P^1$ over the point $0\in \mathbb P^1$ is denoted by $\mathbb P^1[r]$. The root stack $\mathbb P^1[r]$ is the weighted projective line with a single stack point of order $r$ at $0$. We will be dealing with this stack when we study stationary Gromov-Witten theory of curves in Section \ref{section-stationary}.
\end{example}

The evaluation maps for orbifold Gromov-Witten invariants land on the inertia stack of the target orbifold. The coarse moduli space $\underline{I}X_{D,r}$ of the inertia stack of the root stack $X_{D,r}$ can be decomposed into disjoint union of $r$ components
\[
\underline{I}X_{D,r}=X\sqcup \coprod_{i=1}^{r-1} D,
\]
where there are $r-1$ components isomorphic to $D$. The component $X$ is called the identity component. Other components are called twisted sectors. 

The partition $\vec k$ can be used to impose orbifold data of orbifold stable maps as follows. We assume that $r>k_i$, for all $1\leq i\leq m$. For orbifold invariants of the root stack $X_{D,r}$, we consider the moduli space $\overline{M}_{g,\vec k,n,d}(X_{D,r})$ of $(m+n)$-pointed, genus g, degree $d$, orbifold stable maps to $X_{D,r}$ whose orbifold data is given by the partition $\vec k$, such that
\begin{itemize}
\item for $1\leq i\leq m$, the coarse evaluation map $\on{ev}_i$ at the $i$-th marked point lands on the twisted sector $D$ with age $k_i/r$. These marked points are orbifold marked points.
\item the coarse evaluation maps $\on{ev}_i$ at the last $n$ marked points all land on the identity component $X$ of the coarse moduli space of the inertia stack $IX_{D,r}$. These marked points are non-orbifold marked points.
\end{itemize}

Orbifold Gromov-Witten invariants of $X_{D,r}$ are defined as

\begin{align}\label{orbifold-invariant-higher-dimension}
&\left\langle \prod_{i=1}^m\tau_{a_{i}}(\delta_{i})\prod_{i=1}^n \tau_{a_{m+i}}(\gamma_{m+i})\right\rangle^{X_{D,r}}_{g,\vec k,n,d}:=\\
\notag &\int_{[\overline{M}_{g,\vec k,n,d}(X_{D,r})]^{vir}}\bar{\psi}_1^{a_1}\on{ev}^*_{1}(\delta_{1})\cdots \bar{\psi}_m^{a_m}\on{ev}^*_{m}(\delta_{m})\bar{\psi}_{m+1}^{a_{m+1}}\on{ev}^{*}_{m+1}(\gamma_{m+1})\cdots\bar{\psi}_{m+n}^{a_{m+n}}\on{ev}^{*}_{m+n}(\gamma_{m+n}),
\end{align}
where the descendant class $\bar{\psi}_i$ is the class pullback from the corresponding descendant class on the moduli space $\overline{M}_{g,m+n,d}(X)$ of stable maps to $X$.

The basic constructions and fundamental properties of orbifold Gromov-Witten theory can be found in \cite{Abramovich}, \cite{AGV02}, \cite{AGV}, \cite{CR} and \cite{Tseng}.

\subsection{Relations and Questions}

By \cite[Theorem 2]{MP}, relative Gromov-Witten invariants of a smooth pair $(X,D)$ can be uniquely and effectively reconstructed from the Gromov-Witten theory of $X$, the Gromov-Witten theory of $D$, and the restriction map $H^*(X,\mathbb Q)\rightarrow H^*(D,\mathbb Q)$.  On the other hand, for the smooth pair $(X,D)$, we conjectured\footnote{For smooth Deligne-Mumford stacks Y and a smooth divisor $D$, we proved the conjecture when $D$ is disjoint from the locus of stack structures of $X$ \cite{TY16}. The more general version of our conjecture is recently proved by \cite{CDW}.} and proved that the Gromov-Witten theory of root stack $X_{D,r}$ is also determined by the Gromov-Witten theory of $X$, the Gromov-Witten theory of $D$, and the restriction map $H^*(X,\mathbb Q)\rightarrow H^*(D,\mathbb Q)$ \cite{TY16}. This provides another evidence that these two theories may be related.

The relationship between relative and orbifold Gromov-Witten invariants in genus zero has been established by Abramovich-Cadman-Wise \cite{ACW} when the target is a smooth pair $(X,D)$. The relationship was first observed in \cite{CC} for genus zero maps to $X=\mathbb P^2$ with tangency conditions along a smooth plane cubic $D$. It was observed that,  for large and divisible $r$, orbifold Gromov-Witten invariants of the root stack $\mathbb P^2_{D,r}$ stabilize and coincide with relative Gromov-Witten invariants of $(\mathbb P^2,D)$. It was proved in \cite{ACW} that genus zero orbifold Gromov-Witten invariants of $X_{D,r}$ for large and divisible $r$ agree with genus zero relative Gromov-Witten invariants of $(X,D)$ for any $X$ and any $D$. The proof used comparison of virtual fundamental classes of different moduli spaces. 

The goal of this paper is to study the relationship between these relative and orbifold Gromov-Witten invariants in all genera.
In general the result of \cite{ACW} does not hold for higher genus invariants, as shown by a counterexample (due to D. Maulik) for genus one invariants in \cite[Section 1.7]{ACW}. Naturally, we ask the following questions.

\begin{question}\label{question-1}
What is the precise relationship between relative and orbifold Gromov-Witten invariants in higher genus?
\end{question}

\begin{question}\label{question-2}
Will the equality between higher genus relative and orbifold Gromov-Witten invariants hold under some assumptions?
\end{question}

In this paper, we answer the first question for invariants of smooth projective varieties and answer the second question for invariants of target curves.

\subsection{Higher Genus Invariants of General Targets}

For a smooth pair $(X,D)$, the orbifold invariants of $X_{D,r}$ in general depend on $r$. On the other hand, the relative invariants of $(X,D)$ do not depend on $r$. Hence, it is not expected that the exact equality between invariants of $X_{D,r}$ and $(X,D)$ holds in general. The precise relationship is the following:

\begin{theorem}\label{main-theorem}
Given a smooth projective variety $X$, a smooth divisor $D\subset X$, and a sufficiently large integer $r$, the orbifold Gromov-Witten invariant
\[
\left\langle \prod_{i=1}^m\tau_{a_{i}}(\delta_{i})\prod_{i=1}^n \tau_{a_{m+i}}(\gamma_{m+i})\right\rangle^{X_{D,r}}_{g,\vec k,n,d}
\]
of $X_{D,r}$ is a polynomial in $r$. Moreover, 
relative Gromov-Witten invariants of $(X,D)$ are the $r^0$-coefficients of orbifold Gromov-Witten invariants of $X_{D,r}$. More precisely,
\begin{align}
\left\langle \prod_{i=1}^m\tau_{a_{i}}(\delta_{i})\left|\prod_{i=1}^n \tau_{a_{m+i}}(\gamma_{m+i})\right.\right\rangle^{(X,D)}_{g,\vec k,n,d}=\left[\left\langle \prod_{i=1}^m\tau_{a_{i}}(\delta_{i})\prod_{i=1}^n \tau_{a_{m+i}}(\gamma_{m+i})\right\rangle^{X_{D,r}}_{g,\vec k,n,d}\right]_{r^0},
\end{align}
where the notation $[]_{r^0}$ stands for taking the coefficient of $r^0$-term of a polynomial in $r$.
\end{theorem}

\begin{remark}
Theorem \ref{main-theorem} can also be formulated on the cycle level. This is because the techniques that we are using in this paper are the degeneration formula and the virtual localization formula. Both formulas are on the level of virtual cycles. The virtual class version of Theorem \ref{main-theorem} can be proved by straightforward adaptations of the arguments in this paper. In particular, a virtual class version of Theorem \ref{main-theorem} is stated in \cite{FWY18} for genus zero invariants and will appear in \cite{FWY19} for higher genus  invariants. Note that, the results in \cite{FWY18} and \cite{FWY19} extend the result of this paper to include relative invariants with negative contact orders.
\end{remark}

Theorem \ref{main-theorem} directly implies the following result.
\begin{corollary}
The relative Gromov-Witten invariants of $(X,D)$ are completely determined by the orbifold Gromov-Witten invariants of the root stacks $X_{D,r}$ for all sufficiently large $r$.
\end{corollary}

\begin{example}
In genus zero, relative invariants of $(X,D)$ are equal to orbifold invariants of $X_{D,r}$, for $r$ sufficiently large \cite{ACW}. There is a counterexample in genus one given by D. Maulik in \cite[Section 1.7]{ACW}. It is worth to point out that Maulik's counterexample does fit into our result. The example is as follows. Let $X=E\times \mathbb P^1$, where $E$ is an elliptic curve. Consider the divisor $D=X_0\cup X_\infty$, the union of $0$ and $\infty$ fibers of $X$ over $\mathbb P^1$. One can consider the root stack $X_{D, r,s}$ obtained from taking $r$-th root along $X_0$ and $s$-th root along $X_\infty$. One can compare relative invariants of $(X,D)$ and orbifold invariants of the root stack $X_{D, r,s}$. Taking a fiber class $f\in H_2(X)$ of the fibration $X\rightarrow \mathbb P^1$, the genus one relative and orbifold invariants with no insertions are computed in \cite[Section 1.7]{ACW}:
\begin{align*}
&\langle \rangle_{1,f}^{(X,D)}=0;\\
&\langle \rangle_{1,f}^{X_{D, r,s}}=r+s.
\end{align*}
Hence, we have
\begin{align*}
\langle \rangle_{1,f}^{(X,D)}=\left[\langle \rangle_{1,f}^{X_{D, r,s}}\right]_{r^0s^0}.
\end{align*}
\end{example}

The proof of Theorem \ref{main-theorem} follows from degeneration formula and virtual localization computation.

By degeneration formula, we can reduce Theorem \ref{main-theorem} to the comparison between the following invariants of (relative) local models. We can consider the degeneration of $X$ (resp. $X_{D,r}$) to the normal cone of $D$ (resp. $\mathcal D_r$). Indeed, let $Y:=\mathbb P(\mathcal O_D\oplus N)$ where $N$ is the normal bundle of $D\subset X$, we will consider relative invariants of $(Y,D_0\cup D_\infty)$, where $D_0$ and $D_\infty$ are zero and infinity sections respectively. On the other hand, we will consider orbifold-relative invariants of $(Y_{D_0,r},D_\infty)$, where $Y_{D_0,r}$ is the $r$-th root stack of the zero section $D_0$ of $Y$. Theorem \ref{main-theorem} reduces to the comparison between relative invariants of $(Y,D_0\cup D_\infty)$ and orbifold-relative invariants of $(Y_{D_0,r},D_\infty)$.

The relationship between invariants of $(Y,D_0\cup D_\infty)$ and of $(Y_{D_0},D_\infty)$ can be found by $\mathbb C^*$-virtual localization. Localization computation relates both relative invariants of $(Y,D_0\cup D_\infty)$ and orbifold-relative invariants of $(Y_{D_0,r},D_\infty)$ to rubber integrals with the base variety $D$.

A key point for the localization computation is the polynomiality of certain cohomology classes on the moduli space $\overline{M}_{g,n,d}(D)$ of stable maps to $D$ which is proved in \cite[Corollary 11]{JPPZ18}, see Section \ref{section-identity-cycle-classes}. For the relationship between relative and orbifold Gromov-Witten theory of curves, the corresponding result is the polynomiality of certain tautological classes on the moduli space $\overline{M}_{g,n}$ of stable curves proved in \cite[Proposition 5]{JPPZ}.  

We can use the localization computation in the proof of Theorem \ref{main-theorem}, without the need of polynomiality, to provide a new proof of the main theorem of \cite{ACW} in Section \ref{section-genus-zero}. The different behavior between genus zero invariants and higher genus invariants can be seen directly from the difference of their localization computations. 

We restrict our discussions to the case when $X$ is a smooth projective variety, but Theorem \ref{main-theorem} can be extended to the case when $X$ is an orbifold. The key ingredient is the generalization of the polynomiality in \cite{JPPZ18} to orbifolds. When $X$ is a one dimensional orbifold, we only need the orbifold version of  the polynomiality  in \cite{JPPZ}, which has been proved in our previous work \cite{TY16a} on double ramification cycles on the moduli spaces of admissible covers.

\subsection{Stationary Invariants of Target Curves}

We answer Question \ref{question-2} for stationary Gromov-Witten invariants of target curves. 

Gromov-Witten theory of target curves has been completely determined in the trilogy \cite{OP06a}, \cite{OP06b} and \cite{OP} by Okounkov-Pandharipande. Gromov-Witten theory of a target curve $C$ is closely related to Hurwitz theory of enumerations of ramified covers of $C$. The GW/H correspondence proved in \cite{OP06a} showed a correspondence between stationary Gromov-Witten invariants of $C$ and Hurwitz numbers of $C$. The main result of \cite{OP06b} showed that equivariant Gromov-Witten theory of $\mathbb P^1$ is governed by the $2$-Toda hierarchy. The Virasoro constraints for target curves were proven in \cite{OP}, the third part of the trilogy.

Moreover, Gromov-Witten theory of $\mathbb P^1$ can be considered as a more fundamental object than Gromov-Witten theory of a point \cite{OP06b}. The stationary Gromov-Witten invariants of $\mathbb P^1$ arise as Eynard-Orantin invariants \cite{NS}, \cite{BOSS}. As an application, Gromov-Witten theory of a point arises in the asymptotics of large degree Gromov-Witten invariants of $\mathbb P^1$ \cite{NS}, \cite{OP09}.

Now we consider stationary invariants of curves. Let $X=C$ be a smooth projective curve and $q$ be a point in $C$, we consider the following stationary relative invariants of $(C,q)$:
\begin{align}\label{relative-invariant-target-curve}
\langle \prod_{i=1}^n\tau_{a_{m+i}}(\omega)| \vec k\rangle_{g,n,\vec k,d}^{(C,q)}
:=
 \int_{[\overline{M}_{g,n,\vec k,d}(C,q)]^{\on{vir}}}\prod_{i=1}^n\psi_{m+i}^{a_{m+i}}\on{ev}^{*}_{m+i}\omega,
\end{align}
where $\omega \in H^2(C,\mathbb Q)$ denote the class that is Poincar\'e dual to a point.

We consider the root stack $C[r]$ of $C$ by taking $r$-th root along $q$. The stationary orbifold invariants of $C[r]$ are defined as
\begin{align}\label{orbifold-invariant-target-curve}
\langle \prod_{i=1}^n\tau_{a_i}(\omega)\rangle_{g,n,\vec k,d}^{C[r]}
:=
 \int_{[\overline{M}_{g,n,\vec k,d}(C[r])]^{\on{vir}}}\prod_{i=1}^{m}\on{ev}^*_{i}({\mathbf 1}_{k_i/r})\prod_{i=1}^n\bar{\psi}_{m+i}^{a_{m+i}}\on{ev}^{*}_{m+i}\omega,
\end{align}
where ${\mathbf 1}_{k_i/r}$ is the identity class in twisted sector of age $k_i/r$.

\begin{theorem}\label{theorem-stationary}
Let $C$ be a smooth target curve in any genus. When $r$ is sufficiently large, the stationary Gromov-Witten invariants of $(C,q)$ are equal to the stationary Gromov-Witten invariants of the root stack $C[r]$. That is,
\[
(\ref{relative-invariant-target-curve})=(\ref{orbifold-invariant-target-curve}).
\]
\end{theorem}

\begin{remark}
Theorem \ref{theorem-stationary} can be extended slightly by string equations and dilaton equations for Gromov-Witten theory of $(C,q)$ and $C[r]$ with insertions $\tau_0({\mathbf 1})$ and $\tau_1({\mathbf 1})$.
\end{remark}

The proof is based on the degeneration of the target and the equality in genus zero. 

As an application for the equality between stationary invariants. We obtain the GW/H correspondence for orbifold Gromov-Witten invariants of the root stack $C[r_1,\ldots,r_l]$ obtained by taking sufficiently large $r_i$-th root at the point $q_i\in C$ for $1\leq i\leq l$.

\subsection{Further Discussions}

The exact equality between stationary relative invariants of curves and stationary orbifold invariants of curves is in fact a unique feature for Gromov-Witten theory of curves.
The higher dimensional analogy of the equality between stationary invariants of curves is not correct \footnote{In this context, based on the degeneration and localization analysis, a reasonable analogy of stationary invariants for higher dimensional target is to require the restrictions of all cohomological insertions to $D$ vanish.}. It can already be seen from the counterexample given by Maulik in \cite[Section 1.7]{ACW}. The counterexample is about invariants of $X:=E\times \mathbb P^1$, where $E$ is an elliptic curve, with no insertions. These invariants can be viewed as stationary invariants without any insertions. Moreover, the proof for the equality of stationary invariants of curves in Section \ref{section-proof-stationary-1} used degeneration formula to reduce the equality to the case of invariants with no insertions. For Gromov-Witten theory of curves, the equality reduces to the trivial case. It does not reduce to the trivial case beyond Gromov-Witten theory of curves. Indeed, Maulik's counterexample shows that the equality is not true in general. In \cite[Section 1.7]{ACW}, this counterexample is interpreted as a result of the nontriviality of the Picard group of the elliptic curve $E$. 

\subsection{Plan of the Paper}

The paper is organized as follows.

In Section \ref{section-degeneration}, we reduce the comparison between relative and orbifold invariants to (relative) local models by applying degeneration formulas to relative and orbifold invariants. In Section \ref{section-p1-bundle}, we prove Theorem \ref{main-theorem} for local models by virtual localization. Our localization computation is also used in Section \ref{section-genus-zero} to provide a new proof for the equality between genus zero relative and orbifold invariants. In Section \ref{section-stationary}, we present the proof of Theorem \ref{theorem-stationary}. As an easy consequence of Theorem \ref{theorem-stationary}, we extend the GW/H correspondence to stationary orbifold invariants of curves when the root constructions on the curve are taken to be sufficiently large.   

\section*{Acknowledgment}
We would like to thank F. Janda, R. Pandharipande, A. Pixton and D. Zvonkine for sharing the draft of their paper \cite{JPPZ18} with us. We also want to thank Longting Wu for pointing out how to include orbifold/relative descendant classes in Theorem \ref{main-theorem}. F. Y. would like to thank Vincent Bouchard, Qile Chen, Honglu Fan and Zhengyu Zong for helpful discussions. We would also like to thank the anonymous referees for their extremely useful comments and corrections. H.-H. T. is supported in part by NSF grant DMS-1506551. F. Y. is supported by the postdoctoral fellowship of NSERC and Department of Mathematical Sciences at the University of Alberta.

\section{Degeneration}\label{section-degeneration}

In this section, we show that Theorem \ref{main-theorem} and Theorem \ref{theorem-stationary} can be reduced to the case of $\mathbb P^1$-bundles by the degeneration formula. It can be understood by observing that the comparison between relative and orbifold invariants is "local over the divisor $D$", hence it is sufficient to compare invariants of local models. The degeneration formula  gives the precise statement for this observation.

Following \cite{TY16}, we consider the degeneration of $X_{D,r}$ to the normal cone of $\mathcal D_r$, the divisor of $X_{D,r}$ lying over $D\subset X$. The degeneration formula in \cite[Theorem 0.4.1]{AF} shows that orbifold Gromov-Witten invariants of $X_{D,r}$ are expressed in terms of relative Gromov-Witten invariants of $(X_{D,r},\mathcal D_r)$ and of $(\mathcal Y, \mathcal D_\infty)$, where $\mathcal Y:=\mathbb P(\mathcal O \oplus\mathcal N )$ is obtained from the normal bundle $\mathcal N$ of $\mathcal D_r\subset X_{D,r}$; the infinity section $\mathcal D_\infty$ of $\mathcal Y\rightarrow \mathcal D_r$ is identified with $\mathcal D_r\subset X_{D,r}$ under the gluing.

By \cite[Proposition 4.5.1]{AF}, relative Gromov-Witten invariants of $(X_{D,r},\mathcal D_r)$ are equal to relative Gromov-Witten invariants of $(X,D)$ and relative Gromov-Witten invariants of $(\mathcal Y,\mathcal D_\infty)$ are equal to relative Gromov-Witten invariants of $(Y_{D_0,r},D_\infty)$, where $Y:=\mathbb P(\mathcal O\oplus N)$ is obtained from the normal bundle $N$ of $D\subset X$ and $Y_{D_0,r}$ is the root stack of $Y$ constructed by taking $r$-th root along the zero section $D_0$ of $Y\rightarrow D$.

Then, the degeneration formula for the orbifold Gromov-Witten invariants of $X_{D,r}$ is indeed written as
\begin{align}\label{degeneration-orbifold}
&\left\langle \prod_{i=1}^m\tau_{a_{i}}(\delta_{i})\prod_{i=1}^n \tau_{a_{m+i}}(\gamma_{m+i})\right\rangle^{X_{D,r}}_{g,\vec k,n,d}=\\
\notag&\sum  \frac{\prod_i \eta_i}{|\on{Aut}(\mathbf \eta)|} \left\langle\left.\prod_{i=1}^m\tau_{a_{i}}(\delta_{i})\prod_{i\in S}\tau_{a_{m+i}}(\gamma_{m+i})\right |\mathbf \eta\right\rangle^{\bullet,(Y_{D_0,r},D_\infty)}_{g_1,\vec k,|S|,\vec \eta,d_1}\left\langle \mathbf \eta^\vee\left|\prod_{i\not\in S}\tau_{a_{m+i}}(\gamma_{m+i})\right.\right\rangle^{\bullet,(X,D)}_{g_2,\vec \eta,n-|S|,d_2}, 
\end{align}
where $\mathbf \eta^\vee$ is defined by taking the Poincar\'e duals of the cohomology weights of the cohomology weighted partition $\mathbf \eta$; $|\Aut(\mathbf \eta)|$ is the order of the automorphism group $\Aut(\mathbf \eta)$ preserving equal parts of the cohomology weighted partition $\mathbf \eta$. The sum is over all splittings of $g$ and $d$, all choices of $S\subset \{ 1,\ldots,n \}$, and all intermediate cohomology weighted partitions $\mathbf \eta$. The superscript $\bullet$ stands for possibly disconnected Gromov-Witten invariants.

\begin{remark}
The degeneration of $X_{D,r}$ can also be constructed as follows. One can first consider the degeneration of $X$ to the normal cone of $D$. The total space of the degeneration admits a divisor $B$ whose restriction to the general fiber is $D$ and restriction to the special fiber is $D_0$, the zero section of $Y=\mathbb P(\mathcal O_D\oplus N)$. Taking the $r$-th root stack along $B$, we have a flat degeneration of $X_{D,r}$ to $X$ glued together with $Y_{D_0,r}$ along the infinity section $D_\infty\subset Y_{D_0,r}$. It yields the same degeneration formula as in (\ref{degeneration-orbifold}).
\end{remark}

For relative Gromov-Witten invariants of $(X,D)$, we consider the degeneration of $X$ to the normal cone of $D$.  It yields the following degeneration formula of \cite{Li}:
\begin{align}\label{degeneration-relative}
&\left\langle \prod_{i=1}^m\tau_{a_{i}}(\delta_{i})\left| \prod_{i=1}^n \tau_{a_{m+i}}(\gamma_{m+i})\right.\right\rangle^{(X,D)}_{g,\vec k,n,d}=\\
\notag&\sum \frac{\prod_i \eta_i}{|\on{Aut}(\mathbf \eta)|} \left\langle \prod_{i=1}^m\tau_{a_{i}}(\delta_{i})\left|\prod_{i\in S}\tau_{a_{m+i}}(\gamma_{m+i})\right|\mathbf \eta\right\rangle^{\bullet,(Y,D_0\cup D_\infty)}_{g_1,\vec k,|S|,\vec \eta,d_1}\left\langle\mathbf \eta^\vee\left|\prod_{i\not\in S}\tau_{a_{m+i}}(\gamma_{m+i})\right.\right\rangle^{\bullet,(X,D)}_{g_2,\vec \eta,n-|S|,d_2}. 
\end{align}
The sum is also over all intermediate cohomology weighted partitions $\mathbf \eta$ and all splitting of $g$, $d$ and $n$. 

The degeneration formulae (\ref{degeneration-orbifold}) and (\ref{degeneration-relative}) take the same form. Hence, the comparison between orbifold invariants of $X_{D_,r}$ and relative invariants of $(X,D)$ reduces to the comparison between invariants of $(Y_{D_0,r},D_\infty)$ and invariants of $(Y,D_0\cup D_\infty)$. More precisely, it is sufficient to compare the relative invariant
\begin{align}\label{invariant-relative-local-model}
\left\langle \prod_{i=1}^m\tau_{a_{i}}(\delta_{i})\left|\prod_{i=1}^n\tau_{a_{m+i}}(\gamma_{m+i})\right|\mathbf \mu\right\rangle^{(Y,D_0\cup D_\infty)}_{g,\vec k,n,\vec\mu,d}
\end{align} 
of $(Y,D_0\cup D_\infty)$ and the orbifold-relative invariant
\begin{align}\label{invariant-orbifold-local-model}
\left\langle\left. \prod_{i=1}^m\tau_{a_{i}}(\delta_{i})\prod_{i=1}^n\tau_{a_{m+i}}(\gamma_{m+i})\right|\mathbf \mu\right\rangle^{(Y_{D_0,r},D_\infty)}_{g,\vec k,n,\vec\mu,d}
\end{align}
of $(Y_{D_0,r},D_\infty)$, where $\mathbf \mu$ is a cohomology weighted partition of $\int_d[D_\infty]$.

\begin{remark}
By the degeneration formula, we should compare disconnected invariants instead of connected invariants. However, the relationship between disconnected invariants follows from the relationship between connected invariants. Hence, it is sufficient to compare connected invariants. 
\end{remark}

As a result, the comparison can be considered as \emph{local} over the relative/orbifold divisor $D$. The pairs  $(Y_{D_0,r},D_\infty)$ and $(Y,D_0\cup D_\infty)$ can be viewed as (relative) local models of $X_{D_,r}$ and $(X,D)$ . Therefore, Theorem \ref{main-theorem} follows from the following theorem for local models.
\begin{theorem}\label{theorem-local-model}
For $r$ sufficiently large, the orbifold-relative invariant
\[
\left\langle\left.\prod_{i=1}^m\tau_{a_{i}}(\delta_{i})\prod_{i=1}^n\tau_{a_{m+i}}(\gamma_{m+i})\right|\mathbf \mu\right\rangle^{(Y_{D_0,r},D_\infty)}_{g,\vec k,n,\vec \mu,d}
\]
is a polynomial in $r$ and, 
\begin{align}
 \left[\left\langle\left.\prod_{i=1}^m\tau_{a_{i}}(\delta_{i})\prod_{i=1}^n\tau_{a_{m+i}}(\gamma_{m+i})\right|\mathbf \mu\right\rangle^{(Y_{D_0,r},D_\infty)}_{g,\vec k,n,\vec \mu,d}\right]_{r^0}=
 \left\langle \prod_{i=1}^m\tau_{a_{i}}(\delta_{i})\left|\prod_{i=1}^n\tau_{a_{m+i}}(\gamma_{m+i})\right|\mathbf \mu\right\rangle^{(Y,D_0\cup D_\infty)}_{g,\vec k,n,\vec\mu,d}.
\end{align}
\end{theorem}

Similarly, Theorem \ref{theorem-stationary} follows from the following theorem for $(\mathbb P^1[r],\infty)$ and $(\mathbb P^1,0,\infty)$.

\begin{theorem}\label{theorem-stationary-local}
For $r$ sufficiently large, the stationary orbifold-relative invariants of $(\mathbb P^1[r],\infty)$ are equal to the stationary relative invariants of $(\mathbb P^1,0,\infty)$:
\begin{align}\label{identity-stationary}
\langle \mathbf k|\prod_{i=1}^n \tau_{a_{m+i}}(\omega)|\mathbf \mu\rangle^{(\mathbb P^1,0,\infty)}_{g,\vec k,n,\vec\mu,d}=\langle \prod_{i=1}^n \tau_{a_{m+i}}(\omega)|\mathbf \mu\rangle^{(\mathbb P^1[r],\infty)}_{g,\vec k,n,\vec\mu,d}.
\end{align}
\end{theorem}

\begin{remark}
Theorem \ref{theorem-local-model} and Theorem \ref{theorem-stationary-local} can also be stated for disconnected invariants, since the proofs of Theorem \ref{theorem-local-model} and Theorem \ref{theorem-stationary-local} also work for disconnected invariants.
\end{remark}

\section{Local Model}\label{section-p1-bundle}
In this section, we prove Theorem \ref{theorem-local-model} by using virtual localization calculations of \cite{JPPZ18} to obtain identities of cycle classes on moduli spaces.

Let $D$ be a smooth projective variety equipped with a line bundle $L$, and let $Y$ be the total space of the $\mathbb P^1$-bundle
\[
\pi:\mathbb P(\mathcal O_D\oplus L)\rightarrow D.
\] 
Following \cite{MP}, let $e_{1},\ldots,e_{s}$ be a basis of $H^*(D,\mathbb Q)$. We view $e_{i}$ as an element of $H^*(Y,\mathbb Q)$ via pull-back by $\pi$. Let $[D_0], [D_\infty]\in H^2(Y,\mathbb Q)$ denote the cohomology classes associated to the zero and infinity divisors. The cohomological insertions of the invariants will be taken from the following classes in $H^*(Y,\mathbb Q)$:
\[
e_{1},\ldots,e_{s},[D_0]\cdot e_{1},\ldots,  [D_0]\cdot e_{s},[D_\infty]\cdot e_{1},\ldots,[D_\infty]\cdot e_{s}.
\]

We write $Y_{D_0,r}$ for the root stack of $Y$ constructed by taking $r$th root along the zero section $D_0$. The $r$-th root of $D_0$ is denoted by $\mathcal D_r$.

\subsection{Relative Invariants}
Consider the moduli space $\overline{M}_{g,\vec k,n,\vec \mu}(Y,D_0\cup D_\infty)$ of relative stable maps to $(Y,D_0\cup D_\infty)$ with tangency conditions at relative divisor $D_0$ (resp. $D_\infty$) given by the partition $\vec k$ (resp. $\vec \mu$) of $\int_d[D_0]$ (resp. $\int_d[D_\infty]$). The length of $\vec \mu$ is denoted by $l(\mu)$. Recall that the length of $\vec k$ is still denoted by $m$. The following relation between moduli space $\overline{M}_{g,\vec k,n,\vec \mu}(Y,D_0\cup D_\infty)$ of relative stable maps to rigid target and moduli space $\overline{M}_{g,\vec k,n,\vec\mu}(Y,D_0\cup D_\infty)^\sim$ of relative stable maps to non-rigid target is proven in \cite{MP}.
\begin{lemma}[\cite{MP}, Lemma 2]\label{lemma-rigidification}
Let $p$ be a non-relative marking with evaluation map
\[
\on{ev}_p: \overline{M}_{g,\vec k,n,\vec \mu,d}(Y,D_0\cup D_\infty) \rightarrow Y.
\]
Then, the following identities hold.
\begin{align}\label{identity-relative-cycle}
[\overline{M}_{g,\vec k,n,\vec \mu,d}(Y,D_0\cup D_\infty)^\sim]^{\on{vir}}=&\epsilon_*\left(\on{ev}_p^*([D_0])\cap[\overline{M}_{g,\vec k,n,\vec \mu,d}(Y,D_0\cup D_\infty)]^{\on{vir}} \right)\\
\notag =& \epsilon_*\left(\on{ev}_p^*([D_\infty])\cap[\overline{M}_{g,\vec k,n,\vec \mu,d}(Y,D_0\cup D_\infty)]^{\on{vir}} \right),
\end{align}
where
\[
\epsilon: \overline{M}_{g,\vec k,n,\vec \mu,d}(Y,D_0\cup D_\infty)\rightarrow \overline{M}_{g,\vec k,n,\vec \mu,d}(Y,D_0\cup D_\infty)^\sim
\]
is the canonical forgetful map.
\end{lemma}

The proof of Lemma \ref{lemma-rigidification} is through $\mathbb C^*$-localization on the moduli space $\overline{M}_{g,\vec k,n,\vec\mu,d}(Y,D_0\cup D_\infty)$.
 The following identity directly follows from Lemma \ref{lemma-rigidification}.

\begin{lemma}\label{lemma-relative-invariant-rubber}
For $n>0$,
\begin{align}\label{identity-relative-invariant}
&\left\langle \prod_{i=1}^m\tau_{a_{i}}(\delta_{i})\left|\tau_{a_{m+1}}([D_\infty]\cdot \delta_{m+1})\prod_{i=m+2}^{m+n}\tau_{a_i}({\delta_i})\right|\mathbf \mu\right\rangle_{g,\vec k,n,\vec \mu,d}^{(Y,D_0\cup D_\infty)}\\
\notag =&\left\langle \prod_{i=1}^m\tau_{a_{i}}(\delta_i)\left|\prod_{i=m+1}^{m+n}\tau_{a_i}({\delta_i})\right|\mathbf \mu\right\rangle_{g,\vec k,n,\vec \mu,d}^{\sim,(Y, D_0\cup D_\infty)},
\end{align}
where $\delta_{i}\in \pi^*\left(H^*(D,\mathbb Q)\right)$, for $m+1\leq i\leq m+n$, are cohomology classes pulled back from $H^*(D,\mathbb Q)$.
\end{lemma}

\subsection{Orbifold-Relative Invariants}\label{sec:localization-rel-orb}
We use the localization formula of \cite{GP} and \cite{GV} (see also \cite{JPPZ}, \cite{Liu} and \cite{MR}) to study the moduli space $\overline{M}_{g,\vec k,n,\vec \mu}(Y_{D_0,r},D_\infty)$ with prescribed orbifold and relative conditions given by $\vec k$ and $\vec \mu$ respectively. Our goal is to find an identity (see Identity (\ref{identity-orbifold-rubber})) that is similar to Identity (\ref{identity-relative-cycle}), then relates orbifold-relative invariants of $(Y_{D_0,r},D_\infty)$ to rubber integrals as well. 

\subsubsection{The virtual localization formula}

The fiberwise $\mathbb C^*$-action on
\[
\pi:\mathbb P(\mathcal O_D\oplus L)\rightarrow D
\] 
induces a $\mathbb C^*$-action on $Y_{D_0,r}$ and, hence, a $\mathbb C^*$-action on  the moduli space $\overline{M}_{g,\vec k,n,\vec \mu,d}(Y_{D_0,r},D_\infty)$. The class $[\overline{M}_{g,\vec k,n,\vec \mu}(Y_{D_0,r},D_\infty)]^{\on{vir}}$ is computed in \cite[Section 3]{JPPZ18} via the virtual localization formula. For the purpose of our paper, we only need the explicit formula (see Lemma \ref{lemma-localization-formula}). Hence, we will state the formula in this section and refer the readers to \cite[Section 3]{JPPZ} for the derivation of the formula.

The $\mathbb C^*$-fixed loci of $\overline{M}_{g,\vec k,n,\vec \mu,d}(Y_{D_0,r},D_\infty)$ are labeled by decorated graphs. In order to state the virtual localization formula, we need to recall the definition of decorated graphs. We follow \cite{Liu} for the notation of decorated graphs. A decorated graph $\Gamma$ contains the following data.
\begin{itemize}
\item $V(\Gamma)$ is the set of vertices of $\Gamma$. 
Each vertex $v$ is decorated by the genus $g(v)$ and the degree $d(v)\in H_2(D,\mathbb Z)$. The degree $d(v)$ must be an effective curve class. The genus and degree conditions are required
\[
g=\sum_{v\in V(\Gamma)}g(v)+h^1(\Gamma) \quad \text{and} \quad d=\sum_{v\in V(\Gamma)}d(v).
\]
Each vertex $v$ is labeled by $0$ or $\infty$. The labeling map is denoted by 
\[
i:V(\Gamma)\rightarrow \{0,\infty\}.
\] 
\item $E(\Gamma)$ is the set of edges of $\Gamma$. We write $E(v)$ for the set of edges attached to the vertex $v\in V(\Gamma)$ and write $|E(v)|$ for the number of edges attached to the vertex $v\in V(\Gamma)$.
Each edge $e$ is decorated by the degree $d_e\in \mathbb Z_{ >0}$ corresponding to the $d_e$-th power map
\[
\mathbb P^1[r]\rightarrow \mathbb P^1[r].
\]
\item The set of legs is in bijective correspondence with the set of markings. For $1\leq j \leq m$,  the legs are labeled by $k_j\in \mathbb Z_{>0}$ and are incident to vertices labeled $0$. For $m+1\leq j\leq m+n$, the legs are labeled by $0$. For $m+n+1\leq j\leq m+n+l(\mu)$, the legs are labeled by $\mu_{j-m-n}\in \mathbb Z_{>0}$ and are incident to vertices labeled $\infty$. We write $S(v)$ to denote the set of markings assigned to the vertex $v$.

\item The set of flags of $\Gamma$ is defined to be
\[
F(\Gamma)=\{(e,v)\in E(\Gamma)\times V(\Gamma)|v\in e\}.
\]
If the flag is at $0$, then it is labeled by an element $k_{(e,v)}\in \mathbb Z_r$. In fact, in our example, 
\[
k_{(e,v)}=d_e,
\] 
by compatibility along the edge. See, for example, \cite{Johnson}, \cite{Liu} and \cite{JPT}.

\item $\Gamma$ is a connected graph, and $\Gamma$ is bipartite with respect to labeling $i$. Each edge is incident to a vertex labeled by $0$ and a vertex labeled by $\infty$.

\item  A vertex $v\in V(\Gamma)$ is stable if $2g(v)-2+\on{val}(v)> 0$, where $\on{val}(v)$ is the total numbers of marked points and incident edges associated to the vertex $v\in V(\Gamma)$. Otherwise, $v\in V(\Gamma)$ is called unstable. We write $V^S(\Gamma)$ for the set of stable vertices of $\Gamma$. We use $F^S(\Gamma)$ to denote the set of stable flags, that is, the set of flags whose associated vertices are stable.

\item The compatibility condition at a vertex $v$ over $0$:
\begin{align}\label{compatibility-vertex-0}
\sum_{j\in S(v)}k_j-\sum_{e\in E(v)}k_{(e,v)}=\int_{d(v)}c_1(L) \mod r.
\end{align}
The compatibility condition at a vertex is being used in the proof of \cite[Lemma 12]{JPPZ18}, which will be used later in this section.
\item The compatibility condition at a vertex $v$ over $\infty$:
\[
\sum_{e\in E(v)}k_{(e,v)}-\sum_{j\in S(v)}\mu_{j-m-n}=\int_{d(v)}c_1(L).
\]
\end{itemize}

Recall that a vertex $v\in V(\Gamma)$ is unstable if $d(v)=0$ and $2g(v)-2+\on{val}(v)\leq 0$. By \cite[Lemma 12]{JPPZ18}, for $r$ sufficiently large, there are only the following two types of unstable vertices:
\begin{itemize}
\item $v$ is labeled by $0$, $g(v)=0$, $v$ carries one marking and one incident edge;
\item $v$ is labeled by $\infty$, $g(v)=0$, $v$ carries one marking and one incident edge.
\end{itemize}


Following \cite{JPPZ18}, if the target expands at $D_\infty$, the $\mathbb C^*$-fixed locus corresponding to the decorated graph $\Gamma$ is isomorphic to
\[
\overline{M}_{\Gamma}=\prod_{v\in V^S(\Gamma), i(v)=0}\overline{M}_{g(v),\on{val}(v),d(v)}(\mathcal D_r)\times_{D^{|E(\Gamma)|}} \prod_{v\in V^S(\Gamma), i(v)=\infty}\overline{M}_{g(v),\on{val}(v),d(v)}(Y,D_0\cup D_\infty)^{\sim}
\]
quotiented by the automorphism group $\Aut(\Gamma)$ of $\Gamma$ and the product $\prod_{e\in E(\Gamma)}\mathbb Z_{d_e}$ of cyclic groups associated to the edges. 

If the target does not expand, then the moduli spaces of rubber maps do not appear and the invariant locus is the moduli space of stable maps to $\mathcal D_r$. That is,
\[
\overline{M}_{\Gamma}=\overline{M}_{g,m+n+l(\mu),\pi_*d}(\mathcal D_r),
\]
since there is only one vertex over $0$.
The natural morphism
\[
\iota:\overline{M}_{\Gamma}\rightarrow \overline{M}_{g,\vec k,n,\vec \mu,d}(Y_{D_0,r},D_\infty)
\]
is of degree $|\Aut(\Gamma)|\prod_{e\in E(\Gamma)}d_e$. 

The following localization formula is given in \cite[Section 3]{JPPZ18}.
\begin{lemma}\label{lemma-localization-formula}
The virtual localization formula is written as 
\begin{align}\label{localization-formula}
[\overline{M}_{g,\vec k,n,\vec \mu,d}(Y_{D_0,r}, D_\infty)]^{\on{vir}}=
\sum_{\Gamma}\frac{1}{|\on{Aut}(\Gamma)|\prod_{e\in E(\Gamma)}d_e} \cdot\iota_*\left(\frac{[\overline{M}_{\Gamma}]^{\on{vir}}}{e(\on{Norm}_{\Gamma}^{\on{vir}})}\right),
\end{align}
where the sum is taken over decorated graphs $\Gamma$; the inverse of the virtual normal bundle $\frac{1}{e(\on{Norm}_{\Gamma}^{\on{vir}})}$ is the product of the following factors when $r$ is sufficiently large.

\begin{itemize}
\item For each stable vertex $v$ over $0$ in $\Gamma$, there is a factor
\begin{align}
\left(\prod_{e\in E(v)}\frac{rd_e}{t+\on{ev}_{e}^*c_1(L)-d_e\bar{\psi}_{(e,v)}}\right)\cdot\left(\sum_{i=0}^{\infty}(t/r)^{g(v)-1+|E(v)|-i}c_i(-R^*\pi_*\mathcal L)\right),
\end{align}
where 
\[
\pi: \mathcal C_{g(v),\on{val}(v),d(v)}(\mathcal D_r)\rightarrow \overline{M}_{g(v),\on{val}(v),d(v)}(\mathcal D_r)
\]
 is the universal curve, 
\[
\mathcal L\rightarrow \mathcal C_{g(v),\on{val}(v),d(v)}(\mathcal D_r)
\] 
is the universal $r$-th root and $\mathcal O^{(1/r)}$ is a trivial line bundle with a $\mathbb C^*$-action of weight $1/r$.
\item If the target expands over the infinity section, there is a factor
\begin{align}
\frac{\prod_{e\in E(\Gamma)}d_e}{-t-\psi_\infty}.
\end{align}
\end{itemize}
\end{lemma}

\subsubsection{Identity on cycle classes}\label{section-identity-cycle-classes}
\begin{lemma}\label{lemma-orbifold-rubber}
Let $p$ be an interior marking, that is, $p$ is neither a relative marking nor an orbifold marking. For $r$ sufficiently large,
\begin{align}\label{identity-orbifold-rubber}
\left[\epsilon^{\on{orb}}_*\left(\on{ev}_p^*([D_\infty])\cap[\overline{M}_{g,\vec k,n,\vec \mu,d}(Y_{D_0,r}, D_\infty)]^{\on{vir}}\right)\right]_{r^0}=\epsilon^{\on{rel}}_*\left([\overline{M}_{g,\vec k,n,\vec \mu,d}(Y,D_0\cup D_\infty)^\sim]^{\on{vir}} \right),
\end{align}
where $\epsilon^{\on{orb}}$ and $\epsilon^{\on{rel}}$\footnote{ More precisely, $\epsilon^{\on{rel}}$ is the forgetful map to $ \overline{M}_{g,m+n+l(\mu),\pi_*d}(D)$ followed by the inclusion  $\overline{M}_{g,m+n+l(\mu),\pi_*d}(D)\hookrightarrow  \overline{M}_{g,m+n+l(\mu),d}(Y)$.} are forgetful maps 
\[
\epsilon^{\on{orb}}:\overline{M}_{g,\vec k,n,\vec \mu,d}(Y_{D_0,r}, D_\infty) \rightarrow \overline{M}_{g,m+n+l(\mu),d}(Y); 
\]
\[
\epsilon^{\on{rel}}: \overline{M}_{g,\vec k,n,\vec \mu,d}(Y,D_0\cup D_\infty)^\sim\rightarrow  \overline{M}_{g,m+n+l(\mu),d}(Y).
\]
\end{lemma}

\begin{proof}
The localization formula (\ref{localization-formula}) gives
\begin{align}
&\on{ev}_p^*([D_\infty])\cap[\overline{M}_{g,\vec k,n,\vec \mu,d}(Y_{D_0,r}, D_\infty)]^{\on{vir}}=\\
\notag & \qquad \sum_{\Gamma}\frac{1}{|\on{Aut}(\Gamma)|\prod_{e\in E(\Gamma)}d_e} \cdot\iota_*\left( \left(-\on{ev}^*_p(c_1(L))-t\right)\cdot\frac{[\overline{M}_{\Gamma}]^{\on{vir}}}{e(\on{Norm}_{\Gamma}^{\on{vir}})}\right),
\end{align}
where $-\on{ev}^*_p(c_1(L))-t$ is the restriction of the class $[D_\infty]$ to the infinity section $D_\infty$. Following Lemma \ref{lemma-localization-formula}, the inverse of the virtual normal bundle $\frac{1}{e(\on{Norm}_{\Gamma}^{\on{vir}})}$ is the product of the following factors

\begin{itemize}
\item for each stable vertex $v$ over the zero section, there is a factor
\begin{align}
&\left(\prod_{e\in E(v)}\frac{rd_e}{t+\on{ev}_{e}^*c_1(L)-d_e\bar{\psi}_{(e,v)}}\right)\cdot\left(\sum_{i=0}^{\infty}(t/r)^{g(v)-1+|E(v)|-i}c_i(-R^*\pi_*\mathcal L)\right)\\
\notag =& t^{-1}\left(\prod_{e\in E(v)}\frac{d_e}{1+(\on{ev}_{e}^*c_1(L)-d_e\bar{\psi}_{(e,v)})/t}\right)\cdot\left(\sum_{i=0}^{\infty}t^{g(v)-i}(r)^{i-g(v)+1}c_i(-R^*\pi_*\mathcal L)\right)\\
\notag=& t^{-1}\left(\prod_{e\in E(v)}\frac{d_e}{1+(\on{ev}_{e}^*c_1(L)-d_e\bar{\psi}_{(e,v)})/t}\right)\cdot\left(\sum_{i=0}^{\infty}(tr)^{g(v)-i}(r)^{2i-2g(v)+1}c_i(-R^*\pi_*\mathcal L)\right);
\end{align}
\item if the target expands over the infinity section, there is a factor
\begin{align}
\frac{\prod_{e\in E(\Gamma)}d_e}{-t-\psi_\infty}.
\end{align}
\end{itemize}
We consider the pushforward to the moduli space $\overline{M}_{g,m+n+l(\mu),\pi_*d}(D)$ by forgetful maps. Following \cite{JPPZ} and \cite{JPPZ18}, we want to extract the coefficient of $t^0r^0$ from the contributions\footnote{This is slightly different from \cite{JPPZ18}. In \cite{JPPZ18}, they considered localization formula for $\overline{M}_{g,\vec k,n,\vec \mu,d}(Y_{D_0,r}, D_\infty)$ multipled by $t$. In this paper, we consider localization formula for $\on{ev}_p^*([D_\infty])\cap[\overline{M}_{g,\vec k,n,\vec \mu,d}(Y_{D_0,r}, D_\infty)]^{\on{vir}}$, where a factor of $t$ will come from $\on{ev}_p^*([D_\infty])$. In other words, \cite{JPPZ18} wants the coefficient of $t^{-1}$ in the localization formula for $\overline{M}_{g,\vec k,n,\vec \mu,d}(Y_{D_0,r}, D_\infty)$, while we want the coefficient of $t^0$ in the localization formula for $\on{ev}_p^*([D_\infty])\cap[\overline{M}_{g,\vec k,n,\vec \mu,d}(Y_{D_0,r}, D_\infty)]^{\on{vir}}$.}. We set $s:=tr$ and extract $r^0s^0$-coefficient instead. Let
\[
\hat{c}_i=r^{2i-2g+1}\epsilon^{\on{orb}}_* c_i(-R^*\pi_*\mathcal L).
\] 
The inverse of the virtual normal bundle can be rewritten as the product of the factors
\begin{align}\label{contribution-0-pushforward}
\frac rs \prod_{e\in E(v)}\frac{d_e}{1+\frac{r}{s}(\on{ev}_{e}^*c_1(L)-d_e\psi_{(e,v)})}\left(\sum_{i=0}^{\infty}\hat{c}_i s^{g(v)-i}\right), \quad \text{for } v\in V^S(\Gamma)\cap i^{-1}(0);
\end{align}
and 
\begin{align}\label{contribution-infty-pushforward}
-\frac rs\epsilon^{\on{rel}}_*\left(\frac{\prod_{e\in E(\Gamma)}d_e}{1+\frac rs \psi_\infty}\right), \quad \text{if the target expands}.
\end{align}
\cite[Corollary 11]{JPPZ18} states that, for each $i\geq 0$, the class $\hat{c}_i$ is a polynomial in $r$ when $r$ is sufficiently large.

In addition, we have
\[
-\on{ev}^*_p(c_1(L))-t=-\on{ev}^*_p(c_1(L))-\frac sr.
\]
Since the irreducible component containing the non-relative and non-orbifold marked point $p$ maps to $D_\infty$, the target always expands at $D_\infty$. Therefore, there is exactly one factor of  (\ref{contribution-infty-pushforward}) from contributions at $D_\infty$.

Each factor of (\ref{contribution-0-pushforward}) and (\ref{contribution-infty-pushforward}) is of positive power in $r$ and contributes at least one $r$. Therefore, to extract the coefficient of $r^0$, there can be only one such factor, which, of course, has to be  the factor (\ref{contribution-infty-pushforward}) from the only stable vertex over the infinity divisor (there is only one stable vertex over the infinity because there are only unstable vertices over $0$ and the decorated graph is connected). Note that the term $\on{ev}^*_p(c_1(L))$ also disappears, because its product with (\ref{contribution-0-pushforward}) and (\ref{contribution-infty-pushforward}) only produces positive powers of $r$. Therefore, the fixed locus is described by the decorated graph with one stable vertex of full genus $g$ over the infinity section $D_\infty$ and $m$ unstable vertices over the zero section $\mathcal D_r$.

The appearance of higher powers of the target descendant class $\psi_\infty$ in the expansion of (\ref{contribution-infty-pushforward}) will also contribute positive power of $r$, hence the terms involving $\psi_\infty$ are not allowed either. 

Then we extract the coefficient of $s^0$, the result is exactly the right-hand side of (\ref{identity-orbifold-rubber}).

\end{proof}

We consider the invariant
\begin{align}\label{invariant-lemma-orbifold-rubber}
\left\langle \left.\left( \prod_{i=1}^m\tau_{a_{i}}(\delta_i)\right) \tau_{a_{m+1}}([D_\infty]\cdot \delta_{m+1})\prod_{i=m+2}^{m+n}\tau_{a_i}({\delta_i})\right|\mathbf \mu\right\rangle_{g,\vec k,n,\vec \mu,d}^{(Y_{D_0,r}, D_\infty)},
\end{align}
where $\delta_i\in \pi^*\left(H^*(D,\mathbb Q)\right)$, for $m+1\leq i\leq m+n$, are cohomology classes pulled back from $H^*(D,\mathbb Q)$. We have the following relation between orbifold-relative invariants of $(Y_{D_0,r},D_\infty)$ and rubber integrals.

\begin{lemma}\label{lemma-orbifold-invariant-rubber}
For $r$ sufficiently large and $n>0$,
 the orbifold-relative Gromov-Witten invariant (\ref{invariant-lemma-orbifold-rubber}) of $(Y_{D_0,r}, D_\infty)$  is a polynomial in $r$. Moreover,
\begin{align}\label{identity-invariant-orbifold-rubber}
&\left[\left\langle \left. \left( \prod_{i=1}^m\tau_{a_{i}}(\delta_i)\right)\tau_{a_{m+1}}([D_\infty]\cdot \delta_{m+1})\prod_{i=m+2}^{m+n}\tau_{a_i}({\delta_i})\right|\mathbf \mu\right\rangle_{g,\vec k,n,\vec \mu,d}^{(Y_{D_0,r}, D_\infty)}\right]_{r^0}\\
\notag=&\left\langle \prod_{i=1}^m\tau_{a_{i}}(\delta_i)\left|\prod_{i=m+1}^{m+n}\tau_{a_i}({\delta_i})\right|\mathbf \mu\right\rangle_{g,\vec k,n,\vec \mu,d}^{\sim,(Y, D_0\cup D_\infty)}.
\end{align}
\end{lemma}

\begin{proof}

Identity (\ref{identity-invariant-orbifold-rubber}) follows from Identity (\ref{identity-orbifold-rubber}) in Lemma (\ref{lemma-orbifold-rubber}) as the invariants are defined by integrating against the virtual fundamental class and pushing forward to a point.

Polynomiality of the invariant (\ref{invariant-lemma-orbifold-rubber}) follows from the localization analysis and the polynomiality of the class $\hat{c}_i$. Indeed, it is sufficient to consider the factor (\ref{contribution-0-pushforward}):
\begin{align}
\frac 1t \prod_{e\in E(v)}\frac{d_e}{1+\frac{1}{t}(\on{ev}_{e}^*c_1(L)-d_e\psi_{(e,v)})}\left(\sum_{i=0}^{\infty}\hat{c}_i (tr)^{g(v)-i}\right).,
\end{align}
as it is the only factor that depends on $r$.
Negative power of $r$ appears only when $i>g(v)$, but the appearance of negative power of $r$ also results in the same negative power of $t$ in the factor. Hence, negative powers of $r$ do not contribute to the coefficient of $t^0$.
\end{proof}

Combining Lemma \ref{lemma-relative-invariant-rubber} and Lemma \ref{lemma-orbifold-invariant-rubber}, we obtain the identity between relative invariants of $(Y,D_0\cup D_\infty)$ and orbifold-relative invariants of $(Y_{D_0,r},D_\infty)$ with exactly one class of the form $\tau_a([D_\infty]\cdot \delta)$.

\begin{proposition}\label{proposition-local-model}
For $r$ sufficiently large,
\begin{align}
&\left[\left\langle \left. \left( \prod_{i=1}^m\tau_{a_{i}}(\delta_i)\right)\tau_{a_{m+1}}([D_\infty]\cdot \delta_{m+1})\prod_{i=m+2}^{m+n}\tau_{a_i}({\delta_i})\right|\mathbf \mu\right\rangle_{g,\vec k,n,\vec \mu,d}^{(Y_{D_0,r}, D_\infty)}\right]_{r^0}=\\
\notag &\qquad \qquad \left\langle \prod_{i=1}^m\tau_{a_{i}}(\delta_i)\left|\tau_{a_{m+1}}([D_\infty]\cdot \delta_{m+1})\prod_{i=m+2}^{m+n}\tau_{a_i}({\delta_i})\right|\mathbf \mu\right\rangle_{g,\vec k,n,\vec \mu,d}^{(Y,D_0\cup D_\infty)}.
\end{align}
\end{proposition}

\subsection{Proof of Theorem \ref{theorem-local-model}}\label{sec:proof-thm}

In this section, we complete the proof of  Theorem \ref{theorem-local-model}, hence also complete the proof of Theorem \ref{main-theorem}. A special case of Theorem \ref{theorem-local-model} is already given in Proposition \ref{proposition-local-model}. Indeed, the general case of Theorem \ref{theorem-local-model} can be derived from Proposition \ref{proposition-local-model}. In other words, we claim the following.

\begin{lemma}\label{lemma-prop-imply-thm}
All relative Gromov--Witten invariants of $(Y,D_0\cup D_\infty)$ in (\ref{invariant-relative-local-model}) and all relative-orbifold Gromov--Witten invariants of $(Y_{D_0,r}, D_\infty)$  in (\ref{invariant-orbifold-local-model}) satisfy the same universal formulas in which they are determined by invariants of the form in Proposition \ref{proposition-local-model}.
\end{lemma}

We need to prove the identity for the following three types of invariants. The following three types of invariants generate all Gromov--Witten invariants of interest following the description of the cohoological insertions of the invariants at the beginning of Section \ref{section-p1-bundle}.

\begin{description}
\item [Type I] No descendant insertions of the form $\tau_{a}([D_0]\cdot \delta)$ or $\tau_{a}([D_\infty]\cdot \delta)$, where $a\in \mathbb Z_{\geq 0}$ and $\delta \in H^*(D,\mathbb Q)$.

Suppose $\int_d[D_\infty]\neq 0$, by the divisor equation, we have

\begin{align*}
&\left\langle \prod_{i=1}^m\tau_{a_{i}}(\delta_i)\left|\prod_{i=1}^n\tau_{a_{m+i}}(\delta_{m+i})\right|\mathbf \mu \right\rangle^{(Y,D_0\cup D_\infty)}_{g,\vec k,n,\vec \mu,d}\\
=&\frac{1}{\int_d[D_\infty]}\left\langle \prod_{i=1}^m\tau_{a_{i}}(\delta_i)\left|\tau_0([D_\infty])\prod_{i=1}^n\tau_{a_{m+i}}(\delta_{m+i})\right|\mathbf \mu \right\rangle^{(Y,D_0\cup D_\infty)}_{g,\vec k,n+1,\vec \mu,d}\\
&-\frac{1}{\int_d[D_\infty]}\sum_{j=1}^{n}\left\langle \prod_{i=1}^m\tau_{a_{i}}(\delta_i)\left|\tau_{a_{m+j}-1}([D_\infty]\cdot \delta_{m+j})\prod_{i\in \{1,\ldots,n\}\setminus \{j\}}\tau_{a_{m+i}}(\delta_{m+i})\right|\mathbf \mu\right\rangle^{(Y,D_0\cup D_\infty)}_{g,\vec k,n+1,\vec \mu,d}.
\end{align*}
Applying the divisor equation to the corresponding orbifold-relative invariant of $(Y_{D_0,r},D_\infty)$ yields
\begin{align*}
&\left\langle\left.\left( \prod_{i=1}^m\tau_{a_{i}}(\delta_i)\right)\prod_{i=1}^n\tau_{a_{m+i}}(\delta_{m+i})\right|\mathbf \mu\right\rangle^{(Y_{D_0,r},D_\infty)}_{g,\vec k,n,\vec \mu,d}\\
=&\frac{1}{\int_d[D_\infty]}\left\langle\left.\left( \prod_{i=1}^m\tau_{a_{i}}(\delta_i)\right)\tau_0([D_\infty])\prod_{i=1}^n\tau_{a_{m+i}}(\delta_{m+i})\right|\mathbf \mu\right\rangle^{(Y_{D_0,r},D_\infty)}_{g,\vec k,n+1,\vec \mu,d}\\
&-\frac{1}{\int_d[D_\infty]}\sum_{j=1}^{n}\left\langle\left.\left( \prod_{i=1}^m\tau_{a_{i}}(\delta_i)\right)\tau_{a_{m+j}-1}([D_\infty]\cdot \delta_{m+j})\prod_{i\in \{1,\ldots,n\}\setminus \{j\}}\tau_{a_{m+i}}(\delta_{m+i})\right|\mathbf \mu\right\rangle^{(Y_{D_0,r},D_\infty)}_{g,\vec k,n+1,\vec \mu,d}.
\end{align*}

Therefore, the divisor equations for invariants of $(Y,D_0\cup D_\infty)$ and invariants of  $(Y_{D_0,r},D_\infty)$ take the same form. Hence Theorem \ref{theorem-local-model} for invariants of {\bf Type I} follows from Proposition \ref{proposition-local-model} by divisor equations when $\int_d[D_\infty]\neq 0$.

Suppose $\int_d[D_\infty]=0$ and there is at least one non-relative marked point, we may rewrite the relative invariant (\ref{invariant-relative-local-model}) of $(Y,D_0\cup D_\infty)$ as
\begin{align}\label{relative-type-I}
\left\langle \prod_{i=1}^m\tau_{a_{i}}(\delta_i)\left|\prod_{i=1}^n\tau_{a_{m+i}}(\delta_{m+i})\right|\mathbf \mu\right\rangle^{(Y,D_0\cup D_\infty)}_{g,\vec k,n,\vec \mu,d},
\end{align}
where $\delta_{m+i}\in \pi^* H^*(D,\mathbb Q)$ for $1\leq i\leq n$. In this case, decorated graphs in the localization computation do not have edges, hence there is only one vertex. Therefore, the $\mathbb C^*$-fixed locus is just the moduli space rubber maps: $\overline{M}_{g,\vec k,n,\vec \mu,d}(Y,D_0\cup D_\infty)^\sim$. The invariant (\ref{relative-type-I}) is zero because the virtual dimension of the $\mathbb C^*$-fixed locus is $1$ less than the virtual dimension of $\overline{M}_{g,\vec k,n,\vec \mu,d}(Y,D_0\cup D_\infty)$.
Consider the corresponding orbifold invariant of $(Y_{D_0,r},D_\infty)$,
\begin{align}\label{orbifold-type-I}
\left\langle\left. \prod_{i=1}^m\tau_{a_{i}}(\delta_i)\prod_{i=1}^n\tau_{a_{m+i}}(\delta_{m+i})\right|\mathbf \mu\right\rangle^{(Y_{D_0,r},D_\infty)}_{g,\vec k,n,\vec \mu,d}.
\end{align}
Again, the decorated graph has no edge. By the virtual dimension constraint and the localization formula (\ref{localization-formula}), the coefficient of $t^0r^0$ of the invariant (\ref{orbifold-type-I}) is zero. 

Suppose $\int_d[D_\infty]=0$ and there is no non-relative marked point. Choose a class $H\in \pi^*H^2(D,\mathbb Q)$, such that $\int_d H\neq 0$. By divisor equation, this type of invariant can be reduced to the {\bf Type I} invariants with one non-relative marked point of insertion $H$.

Hence we have completed the proof for {\bf Type I} invariants.

\item[Type II] At least one descendant insertions of the form $\tau_{a}([D_\infty]\cdot \delta)$ and no descendant insertions of the form $\tau_{a}([D_0]\cdot \delta)$.

\begin{lemma}\label{lemma-type-I-imply-type-II}
Theorem \ref{theorem-local-model} for invariants of {\bf Type II} follows from the result for invariants of {\bf Type I}.
\end{lemma}

\begin{proof}
We may rewrite the invariant (\ref{invariant-relative-local-model})  of $(Y_{D_0,r},D_\infty)$ as
\begin{align}\label{invariant-orbifold-type-II}
\left\langle \left.\left( \prod_{i=1}^m\tau_{a_{i}}(\delta_i)\right)\prod_{i=1}^{n_0}\tau_{a_{m+i}}(\delta_{m+i})\prod_{i=1}^{n_\infty}\tau_{a_{m+n_0+i}}([D_\infty]\cdot \delta_{m+n_0+i})\right|\mathbf \mu\right\rangle^{(Y_{D_0,r},D_\infty)}_{g,\vec k,n_0+n_\infty,\vec \mu,d}.
\end{align}
We can apply degeneration formula to $(Y_{D_0,r},D_\infty)$ over the infinity divisor $D_\infty$. Hence the invariant (\ref{invariant-orbifold-type-II}) equals to
\begin{align}\label{degeneration-type-II-orbifold-invariant}
&\sum \frac{ \prod_i \eta_i}{|\Aut(\mathbf \eta)|}\left\langle \left.\left( \prod_{i=1}^m\tau_{a_{i}}(\delta_i)\right)\prod_{i\in S}\tau_{a_{m+i}}(\delta_{m+i})\right|\mathbf \eta\right\rangle^{\bullet, (Y_{D_0,r},D_\infty)}_{g_1,\vec k,|S|,\vec \eta,d_1} \cdot\\
\notag & \qquad \left\langle \mathbf \eta^\vee\left|\prod_{i\in \{1,\ldots,n_0\}\setminus S}\tau_{a_{m+i}}(\delta_{m+i})\prod_{i=1}^{n_\infty}\tau_{a_{m+n_0+i}}([D_\infty]\cdot \delta_{m+n_0+i})\right|\mathbf \mu\right\rangle^{\bullet, (Y,D_0\cup D_\infty)}_{g_2,\vec \eta,n_0-|S|+n_\infty,\vec \mu,d_2}.
\end{align}

 The relative invariant of $(Y,D_0\cup D_\infty)$ corresponding to the invariant (\ref{invariant-orbifold-type-II}) is

\begin{align}\label{invariant-relative-type-II}
\left\langle \prod_{i=1}^m\tau_{a_{i}}(\delta_i)\left|\prod_{i=1}^{n_0}\tau_{a_{m+i}}(\delta_{m+i})\prod_{i=1}^{n_\infty}\tau_{a_{m+n_0+i}}([D_\infty]\cdot \delta_{m+n_0+i})\right|\mathbf \mu\right\rangle^{(Y,D_0\cup D_\infty)}_{g,\vec k,n_0+n_\infty,\vec \mu,d}.
\end{align}
Applying the degeneration formula, the invariant (\ref{invariant-relative-type-II}) equals to
\begin{align}\label{degeneration-type-II-relative-invariant}
&\sum \frac{\prod_i \eta_i}{|\Aut(\mathbf \eta)|}\left \langle \prod_{i=1}^m\tau_{a_{i}}(\delta_i)\left|\prod_{i\in S}\tau_{a_{m+i}}(\delta_{m+i})\right|\mathbf \eta\right\rangle^{\bullet, (Y,D_0\cup D_\infty)}_{g_1,\vec k,|S|,\vec \eta,d_1} \cdot\\
\notag & \qquad \left\langle \mathbf \eta^\vee\left|\prod_{i\in \{1,\ldots,n_0\}\setminus S}\tau_{a_{m+i}}(\delta_{m+i})\prod_{i=1}^{n_\infty}\tau_{a_{m+n_0+i}}([D_\infty]\cdot \delta_{m+n_0+i})\right|\mathbf \mu\right\rangle^{\bullet, (Y,D_0\cup D_\infty)}_{g_2,\vec \eta,n_0-|S|+n_\infty,\vec \mu,d_2}.
\end{align}
 The {\bf Type II} orbifold-relative invariants of $(Y_{D_0,r},D_\infty)$ and relative invariants of $(Y,D_0\cup D_\infty)$ satisfy the same form of degeneration formula. Note that the invariants on the first line of (\ref{degeneration-type-II-orbifold-invariant}) and the invariants on the first line of (\ref{degeneration-type-II-relative-invariant}) are of {\bf Type I}. Hence Theorem \ref{theorem-local-model} for invariants of {\bf Type II} follows from the result for {\bf Type I} invariants.
\end{proof}

\item[Type III] At least one descendant insertion of the form $\tau_{a}([D_0]\cdot \delta)$.

The basic divisor relation in $H^2(Y,\mathbb Q)$ gives
\[
[D_\infty]=[D_0]-c_1(L).
\]
Using this formula, invariants of {\bf Type III} can be written as sum of invariants of {\bf Type I} and {\bf Type II}. Hence Theorem \ref{theorem-local-model} for invariants of {\bf Type III} follows from Theorem \ref{theorem-local-model} for {\bf Type I} and {\bf Type II} invariants.
\end{description}

It is straightforward to see that the polynomiality of the orbifold-relative invariant (\ref{invariant-orbifold-local-model}) of $(Y_{D_0,r},D_\infty)$ follows from the above discussion. Indeed, the polynomiality eventually reduces to the polynomiality for the invariants in Proposition \ref{proposition-local-model}, as the universal formulas that we have described in {\bf Type I}, {\bf Type II}  and {\bf Type III} are polynomials, so they preserve the polynomiality. The polynomiality for invariants in Proposition (\ref{proposition-local-model}) is proved in Lemma \ref{lemma-orbifold-invariant-rubber}.

The proof of Theorem \ref{theorem-local-model} is completed.

\section{Genus Zero Relative and Orbifold Invariants}\label{section-genus-zero}

It is proved in \cite{ACW} that relative invariants of $(X,D)$ and orbifold invariants of $X_{D,r}$ are equal in genus zero, provided that $r$ is sufficiently large. The proof in \cite{ACW} is through comparison between virtual fundamental classes on different moduli spaces. In this section we give a new proof for the exact equality between genus zero relative invariants of $(X,D)$ and genus zero orbifold invariants of the root stack $X_{D,r}$ for sufficiently large $r$. Our new proof is through degeneration formula and virtual localization. The reason why the equality fails to hold for higher genus invariants can be seen directly from the localization computation.

We consider the following genus zero relative and orbifold invariants.
\begin{align}\label{relative-invariant-genus-zero}
&\left\langle \prod_{i=1}^m\tau_{a_i}(\delta_i)\left|\prod_{i=1}^n \tau_{a_{m+i}}(\gamma_{m+i})\right.\right\rangle^{(X,D)}_{0,\vec k,n,d}:=\\
\notag &\int_{[\overline{M}_{0,\vec k,n,d}(X,D)]^{vir}}\psi_1^{a_1}\on{ev}^*_{1}(\delta_1)\cdots \psi_m^{a_m}\on{ev}^*_{m}(\delta_m)\cdot\psi_{m+1}^{a_{m+1}}\on{ev}^{*}_{m+1}(\gamma_{m+1})\cdots\psi_{m+n}^{a_{m+n}}\on{ev}^{*}_{m+n}(\gamma_{m+n}),
\end{align}
and
\begin{align}\label{orbifold-invariant-genus-zero}
&\left\langle \prod_{i=1}^m\tau_{a_{i}}(\delta_i)\prod_{i=1}^n \tau_{a_{m+i}}(\gamma_{m+i})\right\rangle^{X_{D,r}}_{0,\vec k,n,d}:=\\
\notag &\int_{[\overline{M}_{0,\vec k,n,d}(X_{D,r})]^{vir}}\bar{\psi}_1^{a_1}\on{ev}^*_{1}(\delta_1)\cdots \bar{\psi}_m^{a_m}\on{ev}^*_{m}(\delta_m)\cdot \bar{\psi}_{m+1}^{a_{m+1}}\on{ev}^{*}_{m+1}(\gamma_{m+1})\cdots\bar{\psi}_{m+n}^{a_{m+n}}\on{ev}^{*}_{m+n}(\gamma_{m+n}).
\end{align}

\begin{theorem}[\cite{ACW}, Theorem 1.2.1]\label{thm:genus-zero}
For $r$ sufficiently large, genus zero relative and orbifold invariants coincide:
\[
(\ref{relative-invariant-genus-zero})=(\ref{orbifold-invariant-genus-zero}).
\]
\end{theorem}

Degeneration formulas in Section \ref{section-degeneration} shows that it is sufficient to prove equality between genus zero invariants of $(Y_{D_0,r},D_\infty)$ and genus zero invariants of $(Y, D_0\cup D_\infty)$. Following the same procedure of Section \ref{section-p1-bundle}, we first prove the following identity on cycles classes in genus zero.

\begin{lemma}\label{lemma-genus-zero}
Let $p$ be a non-orbifold and non-relative marked point. For $r$ sufficiently large, we have
\begin{align}\label{identity-genus-zero}
\epsilon^{\on{orb}}_*\left(\on{ev}_p^*([D_\infty])\cap[\overline{M}_{0,\vec k,n,\vec \mu,d}(Y_{D_0,r}, D_\infty)]^{\on{vir}}\right)\cong \epsilon^{\on{rel}}_*\left([\overline{M}_{0,\vec k,n,\vec \mu,d}(Y,D_0\cup D_\infty)^\sim]^{\on{vir}}\right).
\end{align}
\end{lemma}
\begin{proof}
Following the proof of Lemma \ref{lemma-orbifold-rubber}, the localization formula is
\begin{align}\label{localization-genus-zero}
&\on{ev}_p^*([D_\infty])\cap[\overline{M}_{0,\vec k,n,\vec \mu,d}(Y_{D_0,r}, D_\infty)]^{\on{vir}}=\\
\notag & \qquad \sum_{\Gamma}\frac{1}{|\on{Aut}(\Gamma)|\prod_{e\in E(\Gamma)}d_e} \cdot\iota_*\left( \left(-\on{ev}^*_p(c_1(L))-t\right)\cdot\frac{[\overline{M}_{\Gamma}]^{\on{vir}}}{e(\on{Norm}_{\Gamma}^{\on{vir}})}\right).
\end{align}
The inverse of the virtual normal bundle $\frac{1}{e(\on{Norm}_{\Gamma}^{\on{vir}})}$ can be written as the product of the following factors:
\begin{itemize}
\item for each stable vertex $v$ over the zero section, there is a factor
\begin{align}\label{vertex-contri-genus-zero}
&\prod_{e\in E(v)}\frac{rd_e}{t+\on{ev}_{e}^*c_1(L)-d_e\bar{\psi}_{(e,v)}}\left(\sum_{i=0}^{\infty}(t/r)^{-1+|E(v)|-i}c_i(-R^*\pi_*\mathcal L)\right)\\
\notag =&\left(\frac rt\right)^{|E(v)|}\prod_{e\in E(v)}\frac{d_e}{1+\frac{\on{ev}_{e}^*c_1(L)-d_e\bar{\psi}_{(e,v)}}{t}}\left(\sum_{i=0}^{\infty}(t/r)^{-1+|E(v)|-i}c_i(-R^*\pi_*\mathcal L)\right)\\
\notag =&\frac rt \prod_{e\in E(v)}\frac{d_e}{1+\frac{\on{ev}_{e}^*c_1(L)-d_e\bar{\psi}_{(e,v)}}{t}}\left(\sum_{i=0}^{\infty}(t/r)^{-i}c_i(-R^*\pi_*\mathcal L)\right);
\end{align}
\item if the target expands over the $\infty$-section, there is a factor
\begin{align}
\frac{\prod_{e\in E(\Gamma)}d_e}{-t-\psi_\infty}=-\frac{1}{t}\frac{\prod_{e\in E(\Gamma)}d_e}{1+\frac{\psi_\infty}{t}}.
\end{align}
\end{itemize}
Note that the vertex contribution over the zero section is the corresponding vertex contribution in Lemma \ref{lemma-orbifold-rubber} by setting $g(v)=0$ for all $v$. Therefore, we have the factor $(t/r)^{-i}$ in (\ref{vertex-contri-genus-zero}) instead of the factor $(t/r)^{g(v)-i}$.
As a result, each factor contains only negative powers of $t$ and contributes at least one $t^{-1}$. In order to extract $t^0$-coefficient from (\ref{localization-genus-zero}), there can only be one stable vertex in the decorated graph $\Gamma$. Since the non-orbifold and non-relative marked point $p$ has to land on the infinity divisor $D_\infty$, the only stable vertex is over $\infty$. Therefore, the decorated graph $\Gamma$ is of a stable vertex of genus $0$ over $\infty$ and $m$ unstable vertices over $0$. Since every  $\psi_\infty$ class comes with an extra factor of $t^{-1}$, no term with $\psi_\infty$ class appears in the coefficient of $t^0$. What is left is exactly the right hand side of (\ref{identity-genus-zero}).
\end{proof}

\begin{remark}
The proof does not work for higher genus invariants due to the fact that the contributions from stable vertices over the zero section contain nonnegative power of $t$. Therefore, the coefficient of $t^0$ does not get simplified as in genus zero case. Hence, for higher genus invariants, one needs to pushforward to the moduli space of stable maps to $X$ and also take the coefficient of $r^0$, as discussed in Lemma \ref{lemma-orbifold-rubber}. 
\end{remark}

\begin{proof}[Proof of Theorem \ref{thm:genus-zero}]
By degeneration formulas in Section \ref{section-degeneration}, we only need to compare relative and orbifold invariants of relative local models. Lemma \ref{lemma-genus-zero} implies an equality between genus zero invariants when there is exactly one insertion of the form $\tau_a([D_\infty]\cdot \delta)$ and all other insertions are of the form $\tau_a(\delta)$, where the cohomology class $\delta$ is pulled back from $H^*(D,\mathbb Q)$. In other words, in the genus zero case, the orbifold invariant in Proposition \ref{proposition-local-model} is constant in $r$. Then, Theorem \ref{thm:genus-zero} follows from Lemma \ref{lemma-prop-imply-thm}. In other words, we can follow the same analysis in Section \ref{sec:proof-thm} to prove the general case follows from Proposition \ref{proposition-local-model}. More explicitly, we consider the three types invariants in Section \ref{sec:proof-thm} and restrict the discussion to genus zero invariants. By running through the argument in Section \ref{sec:proof-thm} for genus zero invariants, we see that all these three types of genus zero orbifold invariants are constant in $r$ because the orbifold invariant in Proposition \ref{proposition-local-model} is constant in $r$ when $g=0$.
The proof of Theorem \ref{thm:genus-zero} is completed.  
\end{proof}

\section{Stationary Gromov-Witten Theory of Curves.}\label{section-stationary}

In this section, we prove  Theorem \ref{theorem-stationary-local} for the equality between stationary Gromov-Witten invariants of $(\mathbb P^1[r],\infty)$ and stationary Gromov-Witten invariants of $(\mathbb P^1,0,\infty)$. The proof is based on the degeneration formula in the proof of Theorem \ref{theorem-local-model} and the equality for genus zero invariants. 

\subsection{The Proof of Theorem \ref{theorem-stationary-local}.}\label{section-proof-stationary-1}

By Lemma \ref{lemma-type-I-imply-type-II}, the proof of Theorem \ref{theorem-stationary-local} is reduced to the case of orbifold-relative stationary invariants of $(\mathbb P^1[r],\infty)$ with no stationary marked points, that is,
\begin{align}\label{inv-no-mark}
\langle |\mathbf \mu\rangle^{(\mathbb P^1[r],\infty)}_{g,\vec k,0,\vec \mu,d}. 
\end{align}
This is where we need the invariants to be stationary. More specifically, we consider the degeneration formula (\ref{degeneration-type-II-orbifold-invariant}) in the proof of Lemma \ref{lemma-type-I-imply-type-II} such that all stationary marked points are distributed to the component containing $\infty$. 
There are no insertions in the invariant (\ref{inv-no-mark}), therefore the virtual dimension $\overline{M}_{g,\vec k,0,\vec \mu,d}(\mathbb P^1[r],\infty)$ has to be zero. That is,
 \[
2g-2+m+l(\mathbf \mu)=0.
\] 
This means $g=0$, $m=1$ and $l(\mathbf \mu)=1$. This is genus $0$ invariants of $(\mathbb P^1[r],\infty)$ when there is only one relative marked point, one orbifold marked point and, no non-relative and non-orbifold marked points. 

Similarly for relative invariants of $(\mathbb P^1,0,\infty)$. We only need to consider genus zero invariants of $(\mathbb P^1,0,\infty)$ with single relative marked point at $0$ and $\infty$ respectively; and no non-relative marked points.

Hence, it is sufficient to prove the following equality
\[
\langle \,|(d)\rangle^{(\mathbb P^1[r],\infty)}_{0,(d),0,(d),d}=\langle (d)|\,|(d)\rangle^{(\mathbb P^1,0,\infty)}_{0,(d),0,(d),d},
\]
where $(d)$ represents the trivial partition of $d$ with only one part. It is simply a special case of the equality for genus zero invariants. This completes the proof of Theorem \ref{theorem-stationary-local}.

\begin{remark}
Theorem \ref{theorem-stationary-local} can also be proved by localization comparison which is similar to the proof of Theorem \ref{theorem-local-model}. The key point for the proof of Theorem \ref{theorem-stationary-local} using localization technique is that one can match the vertex contributions using (the disconnected version of) the formulas for double Hurwitz numbers in \cite[Proposition 5.4]{LLZ} and \cite[Theorem 1]{JPT}. Alternatively, one can simply use the formula for double Hurwitz numbers in \cite[Theorem 1]{JPT} to prove that the stationary orbifold invariants are constants in $r$. Hence, by Theorem \ref{main-theorem}, they have to be the same as stationary relative invariants.
\end{remark}

\subsection{Application: Stationary Orbifold Invariants as Hurwitz Numbers}\label{section-stationary-orbifold-hurwitz}
In the celebrated paper \cite{OP06a} by Okounkov-Pandharipande, stationary relative Gromov-Witten invariants of target curves are proven to be equal to Hurwitz numbers with completed cycles, that is, the sum of the Hurwitz numbers obtained by replacing $\tau_{a}(\omega)$ by the associated ramification conditions. The ramification conditions associated to $\tau_{a}(\omega)$ are universal, independent of all factors including the target curve. This is known as GW/H correspondence for relative theory of target curves. Theorem (\ref{theorem-stationary}) states an equality between stationary relative invariants and stationary orbifold invariants of r-th root stacks of target curves. Therefore stationary orbifold Gromov-Witten invariants of r-th root stacks of target curves are equal to Hurwitz numbers with completed cycles when $r$ is sufficiently large. 

We briefly review the theory in \cite{OP06a}. The Hurwitz theory of a smooth curve $C$ describes the enumeration of covers of $C$ with prescribed ramification data given by the cover over the branch points. 

Let $d>0$, and let $\vec \eta^1,\ldots,\vec \eta^l$ be partitions of $d$ assigned to $l$ distinct points $q_1,\ldots, q_l$ of $C$. A Hurwitz cover of $C$ of genus $g$ with ramifications profiles $\vec \eta^1,\ldots, \vec \eta^l$ over $q_1,\ldots,q_l$ is a morphism
\[
\pi:C^\prime\rightarrow C
\]
satisfying the following properties:
\begin{itemize}
\item $C^\prime$ is a nonsingular, connected, genus $g$ curve;
\item the divisors $\pi^{-1}(q_i)$ has ramification profiles equal to the partition $\vec \eta^i,$ for $1\leq i \leq l$;
\item the map $\pi$ is unramified over $C\setminus \{q_1,\ldots,q_l\}$.
\end{itemize}

The Hurwitz number,
\[
H^C_d(\vec \eta^1,\ldots,\vec \eta^l)
\]
is defined to be the weighted count of the distinct Hurwitz covers $\pi$ of genus $g$ with ramifications profiles given by $\vec \eta^1,\ldots,\vec \eta^l$ over $q_1,\ldots,q_l$. Each such cover is weighted by $1/\Aut(\pi)$.

Hurwitz numbers $H^{C}_d(\vec \eta^1,\ldots,\vec \eta^l)$ can be extended to all degree $d$ and all partitions $\vec \eta^i$. Let 
\[
\vec \eta^i=(\eta^i_1,\ldots,\eta^i_{l(i)}),
\] 
and $|\eta^i|=\sum_j^{l(i)} \eta^i_j$, where $l(i)$ is the length of $\vec \eta^i$.
Hurwitz numbers $H^{C}_d(\vec \eta^1,\ldots,\vec \eta^l)$ are defined as follows:
\begin{itemize}
\item $H^{C}_0(\emptyset,\ldots,\emptyset)=1$, where $\emptyset$ stands for empty partition.
\item If $|\vec \eta^i|>d$ for some $i$, then the Hurwitz number vanishes.
\item If $|\vec \eta^i|\leq d$ for all $1\leq i\leq l$, then the Hurwitz number is defined as
\begin{align}
H^{C}_d(\vec \eta^1,\ldots,\vec \eta^l)=\prod_{i=1}^l{m_1(\vec \eta_+^i)\choose m_1(\vec \eta^i)}\cdot H^{C}_d(\vec \eta^1_+,\ldots,\vec \eta^l_+),
\end{align}
where $\vec \eta_+^i$ is the partition of $d$ determined by adjoining $d-|\vec \eta^i|$ parts of size $1$:
\[
\vec \eta_+^i=(\eta^i_1,\ldots,\eta^i_{l(i)},1,\ldots,1);
\]
$m_1(\vec \eta)$ is the multiplicity of the $1$ in $\vec \eta$.
\end{itemize}

Let $S(d)$ be the symmetric group. The class algebra $\mathcal Z(d)\subset \mathbb Q S(d)$ is the center of the group algebra $\mathbb Q S(d)$. Let $C_{\vec \eta}\in \mathcal Z(d)$ be the conjugacy class corresponding to the partition $\vec \eta$. Let $\lambda$ be an irreducible representation of $S(d)$. The conjugacy class $C_{\vec \eta}$ acts as a scalar operator on $\lambda$ with eigenvalue
\[
f_{\vec \eta}(\lambda)={d \choose |\vec \eta|} |C_{\vec \eta}|\frac{\chi_\eta^\lambda}{\dim \lambda},
\]
where $\chi_{\vec \eta}^\lambda$ is the character of any element of $C_{\vec \eta}$ in the representation $\lambda$ and $\dim \lambda$ is the dimension of the representation $\lambda$.

Let $\mathcal P$ be the set of all partitions. There is a linear, injective Fourier transform
\begin{align*}
\phi: \bigoplus_{d=0}^\infty \mathcal Z(d)&\rightarrow \mathbb Q^{\mathcal P}\\
C_{\vec \eta}&\mapsto f_{\vec \eta}.
\end{align*}
The image of $\phi$ is the set of so-called shifted symmetric functions $\Lambda^*$. An element $f$ of the algebra of shifted symmetric functions $\Lambda^*$ can be concretely given as a sequence of polynomials
\[
f=\{f^{(n)}\}, \quad f^{(n)}\in \mathbb Q[\lambda_1,\ldots,\lambda_n]^{*S(n)},
\]
where $\mathbb Q[\lambda_1,\ldots,\lambda_n]^{*S(n)}$ is the invariants of the shifted action of the symmetric group $S(n)$ on the algebra $\mathbb Q[\lambda_1,\ldots,\lambda_n]$. The shifted action is defined by permutation of the variables $\lambda_i$. The sequence $\{f^{(n)}\}$ satisfies
\begin{itemize}
\item $f^{(n)}$ are of uniformly bounded degree,
\item $f^{(n)}$ are stable under restriction, that is,
$f^{(n+1)}|_{\lambda_{n+1}=0}=f^{(n)}$.
\end{itemize}
The shifted symmetric power sum $p_k\in \Lambda^*$ is defined by
\[
p_k(\lambda)=\sum_{i=1}^\infty \left[(\lambda_i-i+\frac 12)^k-(-i+\frac 12)^k\right]+(1-2^{-k})\zeta(-k).
\]

For each partition $\vec \eta$, define $p_{\vec \eta}\in \Lambda^*$ as
\[
p_{\vec \eta}=\prod p_{\eta_i}.
\] 
The completed conjugacy classes are defined by
\[
\overline{C}_{\vec \eta}=\frac{1}{\prod_i \eta_i}\phi^{-1}(p_{\vec \eta})\in \bigoplus_{d=0}^{|\eta|}\mathcal Z(d).
\]
The completed cycles are defined by
\[
\overline{(a)}=\overline{C}_{(a)}, \quad a=1,2,\ldots.
\]
More concretely, completed cycle $\overline{(a)}$ is obtained from the cycle $(a)$ by adding  multiples of constant terms and nonnegative multiples of nontrivial conjugacy classes of strictly smaller size. More details can be found in \cite[Section 0.4]{OP06a}. 

The following GW/H correspondence is proved in \cite{OP06a}:
\begin{theorem}(\cite{OP06a}, Theorem 1)\label{theorem-GW/H-correspondence}
Let $C$ be a smooth target curve of any genus. The GW/H correspondence for the relative Gromov-Witten theory of $C$ is
\begin{align}
\langle \prod_{i=1}^n\tau_{a_i}(\omega)| \eta^1|\ldots|\eta^l\rangle_{g,n,\vec \eta^1,\ldots,\vec \eta^l,d}^{\bullet, (C,q_1,\ldots,q_l)}=\frac{1}{\prod (a_i!)}H_d^{C}(\overline{(a_1+1)},\ldots, \overline{(a_n+1)},\vec \eta^1,\ldots,\vec \eta^l),
\end{align}
where $\omega \in H^2(C,\mathbb Q)$ is the Poincar\'e dual of the point class.
\end{theorem}

Theorem \ref{theorem-stationary} and Theorem \ref{theorem-GW/H-correspondence} together imply the following GW/H correspondence for orbifolds.

\begin{corollary}\label{GW/H-orbifold}
Let $C$ be a smooth target curve in any genus. Let $C[r_1,\ldots,r_l]$ be the root stack over $C$ by taking $r_i$-th root at  the point $q_i\in C$, for the $l$ distinct points $q_1,\ldots,q_l$ of $C$. When $r_i$ are sufficiently large for all $1\leq i\leq l$, we have the following GW/H correspondence:
\begin{align}
\langle \prod_{i=1}^n\tau_{a_i}(\omega)\rangle_{g,n,\vec \eta^1,\ldots,\vec \eta^l,d}^{\bullet, C[r_1,\ldots,r_l]}=\frac{1}{\prod a_i!}H_d^{C}(\overline{(a_1+1)},\ldots, \overline{(a_n+1)},\vec \eta^1,\ldots,\vec \eta^l),
\end{align}
where $\omega \in H^2(C,\mathbb Q)$ is the Poincar\'e dual of the point class.
\end{corollary}

\end{document}